\documentclass[11pt]{amsart}

\usepackage{latexsym}
\usepackage{amsmath}
\usepackage{amsfonts}
\usepackage{amssymb}
\usepackage{enumitem}
\usepackage{blkarray}
\usepackage {tikz}
\usepackage{color}
\usepackage{mathtools}
\usetikzlibrary {positioning}
\usepackage{blkarray, bigstrut}
\usetikzlibrary{arrows,shapes}
\usepackage{color}

\newtheorem{theorem}{Theorem}[section]

\newtheorem{corollary}[theorem]{Corollary}
\newtheorem{proposition}[theorem]{Proposition}
\newtheorem{sublemma}{Step} [theorem]

\theoremstyle{definition}

\theoremstyle{remark}
\newtheorem{remark}[theorem]{Remark}

\numberwithin{equation}{section}
\usepackage{stackengine}

\stackMath

\newcommand\supp{\operatorname{supp}}

\begin{document}

\title[Loose elements]{Loose elements in binary and ternary matroids}

\author{Jagdeep Singh}
\address{Department of Mathematics and Statistics\\
Mississippi State University\\
Starkville, Mississippi 39762}
\email{singhjagdeep070@gmail.com}

\author{Thomas Zaslavsky}
\address{Department of Mathematics and Statistics\\
Binghamton University\\
Binghamton, New York 13902-6000}
\email{zaslav@math.binghamton.edu}

\date{\today}

\begin{abstract}
We call a matroid element ``loose'' if it is contained in no circuits of size less than the rank of the matroid. A matroid in which all elements are loose is a paving matroid. 
Acketa determined all binary paving matroids, while Oxley specified all ternary paving matroids. We characterize the binary matroids that contain a loose element. For ternary matroids with a loose element, we show that their size is linear in terms of their rank. 
Moreover, for a prime power $q$, we give a partial characterization of $GF(q)$-representable matroids that have two or more loose elements; we note Rajpal's partial characterization of $GF(q)$-representable paving matroids as a consequence. 
\end{abstract}

\maketitle

\section{Introduction}
\label{intro}
An element $t$ of a matroid $M$ is called \textbf{free} if every circuit of $M$ that contains $t$ is spanning. We say $t$ is \textbf{loose} if every circuit of $M$ containing $t$ has size at least the rank $r(M)$ of the matroid $M$. A matroid $M$ is called \textbf{paving} if every circuit of $M$ has size at least $r(M)$. It is clear that every element of a paving matroid is loose. Oxley and Singh \cite[Lemma 3.1]{oxjag}  characterized binary and ternary matroids that have a free element. In Section 2 we characterize the binary matroids that have a loose element. 

All matroids considered here are finite and simple. The notation and terminology follow \cite{text}, except where otherwise indicated. 

We define some necessary matroids.  
Let $\{b, x_i, y_i\}$, $ 1 \leq i \leq r-1$,  be a collection of $r-1$ three-point lines, and let  $P_r$ denote their parallel connection at the common basepoint $b$. Suppose that $\{b,x_1,\ldots,x_{r-1}\}$ forms a basis of $P_r$. The matroid $M_r$ is the binary extension of $P_r$ obtained by adding an element $z$ to $P_r$ such that $\{z, b,x_1,\ldots,x_{r-1}\}$ is a circuit of $M_r$. Note that $M_r$ is self-dual. By further adding an element on the line spanned by $\{z,b\}$ and taking its dual, we obtain $N_{r+1}$. 

The matroid $L_r$ is derived from the Fano matroid by replacing an element with a series class of size $r-2$. Similarly, the matroid $J_r$ is obtained from the binary affine geometry of rank four by replacing an element with a series class of size $r-3$.  

Representations of the matroids $M_r$ and $N_r$ are shown in Figures \ref{Mr} and \ref{Nr}.

\begin{theorem}
\label{almost_loose_binary_intro}
Let $M$ be a simple binary matroid of rank $r$ that has no coloops. Then $M$ has a loose element $e$ if and only if one of the following holds. 

\begin{enumerate}[label={\rm(\roman*)}]

    \item $M$ is isomorphic to $L_r$ or $J_r$, or
    
    \item $M$ is isomorphic to a restriction of $M_r$ or $N_r$ that contains $e$.
\end{enumerate}
\end{theorem}

Observe that circuits are among the matroids in Theorem \ref{almost_loose_binary_intro} as they are restrictions of $M_r$. 

Section 3 considers ternary matroids that have a loose element. We show that the size of such matroids is linear in terms of their rank. The main result is the following. 

\begin{theorem}
\label{ternary_one_aloose_intro}
For $r \geq 5$, let $M$ be a simple ternary matroid of rank $r$ a loose element and no coloops. Then $|E(M)| \leq \lfloor\frac{41r-101}{2} \rfloor$ if $r > 10$. If $r \leq 10$, then $|E(M)| \leq \lfloor\frac{35r-35}{2} \rfloor $. 
\end{theorem}

We may restate the bounds as that $|E(M)| \leq \lfloor \frac{1}{2}\max(41r-101, 35r-35) \rfloor$ for any $r \geq 2$.

A $GF(q)$-matroid is a matroid that is representable over the field of $q$ elements. Rajpal partially characterized $GF(q)$-paving matroids, in which all elements are loose.  His result implies, with some work on low-rank matroids, the characterizations of binary paving matroids by Acketa \cite{acketa} and of ternary paving matroids by Oxley \cite{ternary_pav}. 

\begin{theorem}[Rajpal \cite{rajpal}]
\label{4q_bound_intro}
Let $M$ be a simple $GF(q)$-paving matroid with no coloops that is not a circuit. If $r(M) > q$, then $r(M) \leq 2q$ and $|E(M)| \leq 4q$.
\end{theorem}

In Section 4 we provide a partial characterization of $GF(q)$-matroids that have as few as two loose elements, from which it is easy to deduce Theorem \ref{4q_bound_intro}. The following is the precise statement. 

\begin{theorem}
\label{two_almost_loose_intro}
Let $M$ be a simple $GF(q)$-matroid with no coloops that has two loose elements $e$ and $f$. Then $r(M) \leq 2q$, or $\{e,f\}$ is a cocircuit of $M$.
\end{theorem}

We frequently use a standard representation matrix for a representable matroid $M$ of rank $f$ with respect to a basis $B = \{b_1,b_2,\dots,b_r\}$.  The matrix is $P = [I_r \mid Q]$, where the $i^{\rm th}$ column, $1 \leq i \leq r$, is for the basis element $b_i$ and $f^P$ denotes the column labeled by the element $f$ of $M$.

\begin{figure}
\[
\begin{blockarray}{cccccccc}
 &&&& e &  & &    \\
\begin{block}{[cccc|cccc]}
 \bigstrut[t]&&&& 1 & 1 & 1 & 0  \\
  &&&& 1 & 1 & 0 & 1  \\
 &&&& 1 & 0 & 1 & 1  \\
 &&I_r&& 1 & 0 & 0 & 0  \\
 &&&& \vdots & \vdots & \vdots & \vdots  \\
 &&&& 1 & 0 & 0 & 0 \bigstrut[t ] \\
\end{block}
\end{blockarray}
 \]
 \caption{Representation of $L_r$ over $GF(2)$.}
  \label{Lr}
\end{figure}

\begin{figure}
\[
\begin{blockarray}{cccccccc}
 &&&  & e & &    \\
\begin{block}{[cccc|cccc]}
\bigstrut[t]&&&& 0 & 1 & 1 & 1  \\
 &&&& 1 & 1 & 1 & 0  \\
  &&&& 1 & 1 & 0 & 1  \\
 &&I_r&& 1 & 0 & 1 & 1  \\
 &&&& 1 & 0 & 0 & 0  \\
 &&&& \vdots & \vdots & \vdots & \vdots  \\
 &&&& 1 & 0 & 0 & 0\bigstrut[t ]  \\
\end{block}
\end{blockarray}
 \]
 \caption{Representation of $J_r$ over $GF(2)$.}
  \label{Jr}
\end{figure}

\begin{figure}
\[
\begin{blockarray}{cccccc}
 &&& e &  &    \\
\begin{block}{[ccc|ccc]}
\bigstrut[t ]&&& 1 & \ldots & 1 \\
 &I_r& & \vdots & I_{r-1} &  \\
 &&& 1 &  & \bigstrut[t ] \\
\end{block}
\end{blockarray}
 \]
 \caption{Representation of $M_r$ over $GF(2)$.}
  \label{Mr}
\end{figure}

\begin{figure}
\[
\begin{blockarray}{ccccccc}
&& & e &  &  &    \\
\begin{block}{[ccc|cccc]}
\bigstrut[t]&&& 0 & 1 & \ldots & 1 \\
 &&& 1 & 1 & \ldots & 1 \\
 &I_r&& \vdots & & I_{r-2}  &  \\
 &&& \vdots & & & \\
 &&& 1 & & &\bigstrut[b ]\\
\end{block}
\end{blockarray}
 \]

\caption{Representation of $N_r$ over $GF(2)$.}
\label{Nr}
\end{figure}

\section{loose elements in binary matroids}

Oxley and Singh characterized binary and ternary matroids with no coloops that have a free element. 

\begin{proposition}[\cite{oxjag}]
\label{freedom} 
Let $M$ be a simple $GF(q)$-matroid with no coloops and rank at least two, and let $e$ be a free element of $M$. Then 
\begin{itemize}
\item[(i)] $M$ is a circuit when $q =2$; and
\item[(ii)] when $q = 3$, either $M \cong U_{2,4}$ or $M$ can be obtained from a circuit $C$ containing $e$ by, for some subset $D$ of $C- e$, $2$-summing a copy of $U_{2,4}$ across each element of $D$.
\end{itemize}
\end{proposition}

Now we prove Theorem \ref{almost_loose_binary_intro} that characterizes simple binary matroids without coloops that have a loose element.

\begin{proof}[Proof of Theorem \ref{almost_loose_binary_intro}]
It is clear that both  $L_r$ and $J_r$ have a loose element. If $M$ is a restriction of $M_r$ or $N_r$ with no coloops that contains the element $e$ as represented in Figures \ref{Mr} and \ref{Nr}, then $e$ is a loose element of $M$. 

Conversely, let $M$ be a simple binary matroid of rank $r$ with no coloops that contains a loose element $e$. First suppose that there is a basis $B = \{b_1, \ldots, b_r\}$ of $M$ such that the column $e^P$ in the standard binary representation $P=[I_r \mid Q]$ has all ones. Let $f^P$ be a column in $Q - e^P$ labeled by an element $f$ of $E(M)$. Since $M$ is simple, $f^P$ has at least two ones. Suppose $f^P$ has at least three ones, say at entries $i,j,$ and $k$. Then $(\{e,f\} \cup B) - \{b_i, b_j, b_k\}$ contains a circuit of size at most $r-1$ that includes $e$. Since $e$ is loose, this is a contradiction. Therefore $f^P$ has exactly two ones. Let $f^P, g^P$ be a pair of columns in $Q - e^P$. We show that, if $\{i,j\}$ and $\{k,l\}$ are the non-zero entries of $f^P$ and $g^P$ respectively, then $|\{i,j\} \cap \{k,l\}| = 1$. If not, then $f^P + g^P$ has four ones so $e^P + f^P + g^P$ has four zeroes and $e$ is contained in a circuit of size less than $r$, a contradiction. Therefore $M$ is isomorphic to $L_r$ or a restriction of $M_r$.

We now assume that there is no spanning circuit of $M$ containing $e$, so all circuits of $M$ containing $e$ have size $r$. Let $B$ be basis of $M$ that does not contain $e$. The column $e^P$ in the matrix $P$ has exactly one entry that is zero; we may assume it is the first entry of the column. Let $f^P$ be a column in $Q-e^P$ labeled by an element $f$ of $M$. We show that the first entry of $f^P$ is one and it has exactly three non-zero entries. If the first entry of $f^P$ is zero, then for some $i,j > 1$, the $i^{\rm th}$ and the $j^{\rm th}$ entries of $f^P$ are non-zero. Observe that $(\{e,f\} \cup B) - \{b_1, b_i, b_j\}$ contains a circuit of size at most $r-1$ that includes $e$, a contradiction. Therefore the first entry of $f^P$ is one. It is easy to check that if $f^P$ has exactly two non-zero entries then we get a spanning circuit containing $e$, a contradiction. Also, if $f^P$ has more than three non-zero entries, we get a circuit of size less than $r$ that contains $e$. Therefore $f^P$ has exactly three non-zero entries. Let $f^P, g^P$ be a pair of columns in $Q-e^P$. Note that for the non-zero entries at $\{1,i,j\}$ of $f^P$ and at $\{1,k,l\}$ of $g^P$, we have $|\{i,j\} \cap \{k,l\}| = 1$. If not, then $e^P + f^P + g^P$ has five zeroes so $e$ is in a circuit of size less than $r$, a contradiction. It follows that $M$ is isomorphic to $J_r$ or a restriction of $N_r$. 
\end{proof}

\section{loose elements in ternary matroids}

We show that the size of a ternary matroid that has a loose element is linear in terms of its rank. The main result is as follows.

\begin{theorem}
\label{ternary_one_aloose}
For $r \geq 5$, let $M$ be a simple ternary matroid of rank $r$ containing a loose element and with no coloops. Then $|E(M)| \leq \lfloor\frac{41r-101}{2} \rfloor$ if $r > 10$. If $r \leq 10$, then $|E(M)| \leq \lfloor\frac{35r-35}{2} \rfloor$.
\end{theorem}

\begin{proof}
Let $e$ be a loose element of $M$. If $e$ is free, then the result follows from Proposition \ref{freedom}, so we assume $e$ is not free. Then there is a basis $B = \{e_1, \ldots, e_r\}$ of $M$ such that the column $e^P$ in the standard ternary representation $P = [I_r \mid Q]$ has top entry zero and all the other entries non-zero. By row and column scaling, we may assume that the non-zero entries of $e^P$ are ones. We call the indices of the non-zero entries of a column $f^P$ in $P$ the \textit{support}, $\supp(f^P)$, of $f^P$. Let $Q-e^P$ denote the columns in $Q$ excluding $e^P$. 

\begin{sublemma}
\label{argument}
For $k \geq 1$, let $j^P$ be a vector that is a linear combination of $k$ columns in $Q - e^P$ with coefficients in $\{1, -1\}$, and let $j^{P'}$ be the vector obtained from $j^P$ by removing the top entry. Then $j^{P'}$ has at most $k+1$ entries of the same sign. Moreover, if the top entry of $j^P$ is zero, then the number of entries of the same sign in $j^P$ is at most $k$. 
\end{sublemma}

To prove this, observe that if $j^{P'}$ has more than $k+1$ entries of the same sign, then $e^P + j^P$ or $e^P - j^P$ has at least $k+2$ zeroes. It follows that $e^P + j^P$ or $e^P - j^P$ can be combined with at most $r - (k+2)$ columns in $I_r$ to obtain the zero vector.  Similarly, if the top entry of $j^P$ is zero and $j^P$ has more than $k$ entries of the same sign, then $e^P + j^P$ or $e^P - j^P$ has at least $k+2$ zeroes.  Either way, $e$ is in a circuit of size less than $r$, a contradiction. Therefore \ref{argument} holds.

\begin{sublemma} \label{sublemma_1}
The number of columns in $Q-e^P$ that have top entry zero is at most $\lfloor\frac{r-1}{2} \rfloor$.
\end{sublemma}

For the proof, let $f^P$ be such a column. By \ref{argument}, $f^P$ has exactly two non-zero entries: one $1$ and one $-1$. We show that any two columns $f^P_1, f^P_2$ in $Q-e^P$ with the top entry zero have disjoint supports. It is clear that $f_1^P$ and  $f_2^P$ do not have equal support; otherwise, $f_1$ and $f_2$ are parallel elements in $M$, but $M$ is simple. Suppose the support of $f_1^P$ is $\{i,j\}$ and that of $f_2^P$ is $\{j,k\}$. Then, either $f_1^P+f_2^P$ or $f_1^P-f_2^P$ has three entries of the same sign and has the top entry zero, in contradiction to \ref{argument}. Therefore \ref{sublemma_1} holds. 

\bigskip
Next, consider the columns in $Q-e^P$ that have a non-zero top entry. By row and column scaling, we may assume that the top entry of all such columns is $1$. In the rest of the proof, by \textit{root entries} of such a column $f^P$, we mean the entries of $f^P$ excluding the top entry. By \ref{argument}, it follows that $f^P$ has at most two root entries of the same sign so  $f^P$ has at most four root entries that are non-zero. We call a column $f^P$ with $h$ non-zero root entries a \textit{column of type h}.

\begin{sublemma} \label{sublemma_2}
The number of columns of type $4$ in $Q-e^P$ is at most $8r - 34$.    
\end{sublemma}

Let $f^P$ be a column of type $4$ in $Q-e^P$ with support $\{1,i,j,k,l\}$. We may assume that the $i^{\rm th}$ and the $j^{\rm th}$ entries are $1$ and the $k^{\rm th}$ and the $l^{\rm th}$ entries are $-1$. For any other column $g^P$ of type $4$, we show that the support of $g^P$ intersects both $\{i,j\}$ and $\{k,l\}$. Suppose that the $i^{\rm th}$ and the $j^{\rm th}$ entries of $g^P$ are zero. If $g^P$ has a $-1$ at a position other than $k$ and $l$, then $f^P-g^P$ has three ones and has top entry zero, a contradiction to \ref{argument}. Therefore both $-1$'s of $g^P$ are at positions $k$ and $l$. It now follows that $f^P+g^P$ has four ones, again a contradiction to \ref{argument}. Thus the support of $g^P$ intersects $\{i,j\}$ and by symmetry, $\{k,l\}$ as well. 
Note that the number of columns of type $4$ with the same support as $f^P$ is at most six. 

Suppose $g^P$ is a column of type $4$ whose support intersects $\{i,j,k,l\}$ in three indices.  We may assume that $\supp(g^P) = \{1,i,j,k,m\}$ where $m$ is distinct from $l$. We show that there is only one column with that support. If the $m^{\rm th}$ entry of $g^P$ is $1$, then the $i^{\rm th}$ or the $j^{\rm th}$ entry of $g^P$ is $-1$. It follows that $f^P - g^P$ has top entry zero and has at least three entries that are $-1$. This contradicts \ref{argument}, so the $m^{\rm th}$ entry of $g^P$ is $-1$. If the $k^{\rm th}$ entry of $g^P$ is $-1$, then $f^P + g^P$ has four entries that are $-1$, contradicting \ref{argument}. Therefore the $k^{\rm th}$ entry of $g^P$ is $1$. Let $h^P$ be another column in $Q-e^P$, distinct from $g^P$ but with the same support as $g^P$. The $m^{\rm th}$ entry and the $k^{\rm th}$ entry of $h^P$ are $-1$ and $1$ respectively. Since $h^P$ is distinct from $g^P$, the $i^{\rm th}$ entry of $h^P$ is the negative of the $i^{\rm th}$ entry of $g^P$, and the $j^{\rm th}$ entry of $h^P$ is the negative of the $j^{\rm th}$ entry of $g^P$. Then $f^P + g^P + h^P$ has top entry zero and four entries that are $1$, contradicting \ref{argument}, so no such column $h^P$ can exist. Since there are $r-5$ choices for $m$, the number of columns $g^P$ such that $\supp(g^P) \cap \{i,j,k,l\}$ equals $\{i,j,k\}$ is at most $r-5$. By symmetry this holds for each of the four possible intersections, $\{i,j,k\}, \{i,j,l\}, \{i,l,k\},$ and $\{j,k,l\}$. Therefore, the number of columns $g^P$ such that the intersection of $\supp(g^P)$ with $\{i,j,k,l\}$ is three is at most $4(r-5)$. 

Next, we consider a column $g^P$ of type $4$ in $Q-e^P$ such that the size of $\supp(g^P) \cap \{i,j,k,l\}$ is two. The non-zero entries of $g^P$ in $\{i,j,k,l\}$ have the same sign as in $f^P$, for if not, then $f^P - g^P$ has at least five non-zero entries and has the top entry zero so $f^P-g^P$ has three entries of the same sign, in contradiction to \ref{argument}. The possible intersections of the support of $g^P$ with $\{i,j,k,l\}$ are $\{i,k\}, \{i,l\}, \{j,k\},$ and $\{j,l\}$. 
Let $g_1^P, \ldots, g_t^P$ be the columns of type $4$ in $Q-e^P$ whose support intersects $\{i,j,k,l\}$ at $\{i,k\}$. 
We show that the intersection of the support of any two of $g_1^P, \ldots, g_t^P$ that do not have the same support is $\{i,k\}$. Suppose the supports of $g_1^P$ and $g_2^P$ intersect in $\{i,k,m\}$. If the $m^{\rm th}$ entries of $g_1^P$ and $g_2^P$ are of opposite sign, then $g_1^P- g_2^P$ has three entries of same sign and the top entry is zero, a contradiction to \ref{argument}. Therefore the $m^{\rm th}$ entries of $g_1^P$ and $g_2^P$ are the same. It now follows that $g_1^P + g_2^P$ has four entries of same sign, again contradicting \ref{argument}. Since for a support $s$, there are at most two columns in $\{g_1^P, \ldots, g_t^P\}$ that have the support $s$, it follows that $t$ is at most $2\lfloor\frac{r-5}{2} \rfloor \leq r-5$. By symmetry this holds for each possible intersection $\{i,k\}, \{i,l\}, \{j,k\},$, $\{j,l\}$, so the number of columns $g^P$ of type $4$ in $Q-e^P$ such that $|\supp(g^P) \cap \{\ i,j,k,l\}| = 2$ is at most $4(r-5)$. 

Therefore the number of type $4$ columns is at most $4(r-5) + 4(r-5) + 6 = 8r - 34$.

\begin{sublemma} \label{sublemma_3}
The number of columns of type $3$ in $Q-e^P$ is at most $12r-42$. 
\end{sublemma}

Let $f^P$ be a column of type $3$ in $Q-e^P$ with support $\{1,i,j,k\}$. For $\{\alpha, \beta \} = \{1, -1\}$, we may assume that the $i^{\rm th}$ and the $j^{\rm th}$ entries of $f^P$ are $\alpha$, and the $k^{\rm th}$ entry is $\beta$. 
The support of a column $g^P$ of type $3$ in $Q-e^P$ intersects $\{i,j,k\}$; otherwise, $f^P - g^P$ has at least three entries of the same sign and has the top entry zero, a contradiction to \ref{argument}. 
Also, the number of columns of type $3$ in $Q-e^P$ with the same support as $f^P$ is at most six. 

Consider a column $g^P$ of type $3$ in $Q-e^P$ whose intersection with $\{i,j,k\}$ is two.  The possible intersections are $\{i,j\}, \{j,k\},$ and $\{i,k\}$. 

First, let $g^P$ be a type $3$ column such that the support of $g^P$ is $\{1,i,j,m\}$ where $m$ is distinct from $k$. If the $m^{\rm th}$ entry of $g^P$ is $\alpha$, then the $i^{\rm th}$ or the $j^{\rm th}$ entry of $g^P$ is $\beta$. It follows that $f^P-g^P$ has top entry zero and has three entries that are $\beta$, a contradiction to \ref{argument}. Therefore the $m^{\rm th}$ entry of $g^P$ is $\beta$. Note that at least one of the $i^{\rm th}$ and the $j^{\rm th}$ entries of $g^{\rm th}$ is $\beta$; otherwise, $f^P + g^P$ has four entries that are $\beta$, a contradiction to \ref{argument}. It follows that the $i^{\rm th}$ and $j^{\rm th}$ entries of $g^P$ are of opposite signs. Therefore there are at most two columns of type $3$ that have the support $\{1,i,j,m\}$. Since there are $r-4$ choices for $m$, the number of columns $g^P$ of type $3$ in $Q-e^P$ such that $\supp(g^P) \cap \{i,j,k\}$ equals $\{i,j\}$  is at most $2(r-4)$. 

Second, suppose that $g^P$ is a type $3$ column in $Q-e^P$ such that the support of $g^P$ is $\{1,j,k,m\}$ where $m$ is distinct from $i$. There are six possible columns of type $3$ that have the support $\{1,j,k,m\}$. However, if the $i^{\rm th}$, $k^{\rm th}$, and $m^{\rm th}$ entries of $g^P$ are $\alpha$, $\alpha,$ and $\beta$ respectively, then $f^P-g^P$ has top entry zero and has three entries that are $\alpha$, a contradiction to \ref{argument}. Similarly, if the $i^{\rm th}$, $k^{\rm th}$, and $m^{\rm th}$ entries of $g^P$ are $\beta$, $\alpha,$ and $\beta$ respectively, then $f^P-g^P$ has top entry zero and has three entries that are $\alpha$, another contradiction to \ref{argument}. Therefore, at most four columns of type $3$ in $Q-e^P$ have support $\{1,j,k,m\}$. Since there are $r-4$ choices for $m$, the number of columns $g^P$ in $Q-e^P$ such that $\supp(g^P) \cap \{i,j,k\}$ equals $\{j,k\}$  is at most $4(r-4)$. By symmetry, the number of columns of type $3$ in $Q-e^P$ whose intersection with $\{i,j,k\}$ is $\{i,k\}$ is $4(r-4)$ as well. It follows that the number of columns of type $3$ in $Q-e^P$ whose intersection with $\{i,j,k\}$ has size two is $10(r-4)$. 

Now consider a column $g^P$ such that the intersection of $\supp(g^P)$ with $\{i,j,k\}$ has size one. 

Suppose $\supp(g^P) \cap \{i,j,k\} = \{k\}$. Let $\supp(g^P) = \{1,k,m,n\}$. The $k^{\rm th}$ entry of $g^P$ must be $\beta$, or else $f^P-g^P$ has top entry zero and has three entries that are $\alpha$, contrary to \ref{argument}. Moreover, the $m^{\rm th}$ and the $n^{\rm th}$ entries of $g^P$ must be $\alpha$, or else $f^P-g^P$ has top entry zero and at least three entries that are $\alpha$, also contrary to \ref{argument}. It now follows that $f^P + g^P$ has five entries that are $\alpha$, again contradicting \ref{argument}. Therefore no such column exists in $Q-e^P$.  

Let $g_1^P, \ldots, g_t^P$ be all the columns of type $3$ in $Q-e^P$ whose support intersects $\{i,j,k\}$ at $\{i\}$. For $g^P$ in $\{g_1^P, \ldots, g_t^P\}$, the $i^{\rm th}$ entry of $g^P$ is $\alpha$, because if not, then $f^P - g^P$ has five non-zeroes and has top entry zero; it follows that $f^P - g^P$ has three entries of the same sign while the top entry is zero, in contradiction to \ref{argument}. Note that $g^P$ has two non-zero entries that are $\alpha$; otherwise, $f^P - g^P$ has three entries that are $\beta$ while the top entry is zero, which is impossible. It follows that for any fixed support $s$, there are at most two columns in $\{g_1^P, \ldots, g_t^P\}$ that have the support $s$. 
Now, we show that the intersection of the supports of any pair of columns in $\{g_1^P, \ldots, g_t^P\}$ that have distinct support is $\{i\}$. Suppose not, and let $\supp(g_1^P) \cap \supp(g_2^P) = \{i,m\}$. If the $m^{\rm th}$ entries of $g_1^P$ and $g_2^P$ are of opposite signs, then $g_1^P - g_2^P$ has three entries of the same sign while the top entry is zero, a contradiction. On the other hand, if the $m^{\rm th}$ entries of $g_1^P$ and $g_2^P$ are both $\alpha$, then $g_1^P + g_2^P$ has four entries that are $\beta$, in contradiction to \ref{argument}. Therefore the $m^{\rm th}$ entries of $g_1^P$ and $g_2^P$ are both $\beta$. It follows that $g_1^P + g_2^P + f^P$ has four entries that are $\alpha$ and has top entry zero, contradicting \ref{argument}. 
We conclude that there are only $\lfloor\frac{r-4}{2} \rfloor$ possibilities for $\{m,n\}$, so that $t$ is at most  $2 \lfloor\frac{r-4}{2} \rfloor \leq r-4$. By symmetry this holds when $\supp(g^P) \cap \{i,j,k\}$ equals $\{j\}$ as well. It follows that the number of columns of type $3$ whose size of intersection with $\{i,j,k\}$ equals one is at most $2(r-4)$. 

Thus, the number of type $3$ columns in $Q-e^P$ is at most $10(r-4) + 2(r-4) + 6 = 12r-42$ and \ref{sublemma_3} holds.

\begin{sublemma} \label{sublemma_4}
The number of columns of type $2$ in $Q-e^P$ is at most $12r-40$. 
\end{sublemma}

First, suppose there is a column $f^P$ of type $2$ in $Q-e^P$ that has both non-zero root entries of the same sign. Let the support of $f^P$ be $\{1,i,j\}$. For any other column $g^P$ of type $2$ in $Q-e^P$, the support of $g^P$ intersects $\{i,j\}$; otherwise, $f^P-g^P$ has the top entry zero and has three entries of the same sign, or $f^P+g^P$ has four entries of the same sign, contradicting \ref{argument}. 
Let $g_1^P, \ldots, g_t^P$ be all the columns of type $2$ in $Q-e^P$ whose support intersects $\{i,j\}$ at $\{i\}$. Since for a given support $s$, there are at most four columns of type $2$ in $Q-e^P$ with support $s$, $t$ is at most $4(r-3)$.  By symmetry, the number of columns of type $2$ in $Q-e^P$ whose support intersects $\{i,j\}$ at $\{j\}$ is at most $4(r-3)$. It follows that the number of type $2$ columns in $Q-e^P$ is at most $8(r-3) + 4$.  Since $r \geq 5$, we have $8r-20 \leq 12r-40$ in this case.

We may now assume that no type $2$ column in $Q-e^P$ has two root entries with the same sign. 
If, for every pair $f^P, g^P$ of type $2$ columns in $Q-e^P$, either $\supp(f^P) = \supp(g^P)$ or $\supp(f^P) \cap \supp(g^P) = \{1\}$, then the number of type $2$ columns is at most $4 \lfloor\frac{r-1}{2} \rfloor$ and the result follows. Therefore there exists a pair of columns $f^P$ and $g^P$ that have distinct supports $\{1,i,j\}$ and $\{1,j,k\}$. 
The $j^{\rm th}$ entries of $f^P$ and $g^P$ are the same, or else $f^P - g^P$ would have three same sign entries and top entry zero, a contradiction to \ref{argument}. Let $d^P$ be another type $2$ column. Note that the support of $d^P$ intersects $\{i,j,k\}$ else $f^P + g^P + d^P$ has four entries of the same sign and has the top entry zero, a contradiction to \ref{argument}. It follows that there are $3(r-4)$ choices for the support of $d^P$. Since for a given support $s$, we have at most four columns of type $2$ with support $s$, the number of type $2$ columns is at most $12(r-4) + 4(2) = 12r - 40$. Therefore \ref{sublemma_4} holds in this case as well. 
\bigskip

The following is a straightforward observation.

\begin{sublemma}
\label{sublemma_5}
The number of type $1$ columns in $Q-e^P$ is at most $2r-2$.
\end{sublemma}

\subsubsection*{The Type 4 case.}
Now suppose that there is a column $f^P$ of type $4$ in $Q-e^P$ with support $\{1,i,j,k,l\}$. We may assume that the $i^{\rm th}$ and $j^{\rm th}$ entries of $f^P$ are $1$, and the $k^{\rm th}$ and $l^{\rm th}$ entries are $-1$. 

Let $g^P$ be a column of type $3$ in $Q-e^P$. 
We show that the support of $g^P$ intersects both $\{i,j\}$ and $\{k,l\}$. Suppose that the $i^{\rm th}$ and $j^{\rm th}$ entries of $g^P$ are zero and let $m$ be any element of $\supp(g^P) - \supp(f^P)$. The $m^{\rm th}$ entry of $g^P$ is $1$; otherwise, $f^P-g^P$ has three ones and has top entry zero, a contradiction to \ref{argument}. It follows that the support of $g^P$ is $\{1,k,l,m\}$; else $f^P + g^P$ has at least four ones, contradicting \ref{argument}. Moreover, the $k^{\rm th}$ entry and $l^{\rm th}$ entries of $g^P$ must be $-1$; otherwise, $f^P-g^P$ has at least three ones and top entry zero, a contradiction. It now follows that $f^P+g^P$ has four ones, contradicting \ref{argument}. Thus, the support of $g^P$ intersects $\{i,j\}$ and by symmetry, $\{k,l\}$ as well. 
There are at most $24$ type $3$ columns $g^P$ such that $\supp(g^P) \subseteq \{1,i,j,k,l\}$.  (There are four choices for its supporting root elements.  The root signs cannot be all the same, so there are six ways to choose those signs.)

We next consider those columns $g^P$ such that $|\supp(g^P) \cap \{i,j,k,l\}| = 2$. The possible ways $\supp(g^P)$ intersects with $\{i,j,k,l\}$ are $\{i,k\}, \{i,l\}, \{j,k\},$ and $\{j,l\}$. 
Suppose that $\supp(g^P) = \{1,j,l,m\}$, so $ \supp(g^P) \cap \{i,j,k,l\} = \{j,l\}$. 
Observe that the $j^{\rm th}$ entry of $g^P$ is $1$ or the $l^{\rm th}$ entry of $g^P$ is $-1$. If not, then $f^P-g^P$ has the top entry zero and has three entries of the same sign, a contradiction to \ref{argument}. 
Furthermore, the $j^{\rm th}$ and $l^{\rm th}$ entries of $g^P$ cannot be equal.  If both equal $1$, the $m^{\rm th}$ entry of $g^P$ is $-1$.  If both equal $-1$, the $m^{\rm th}$ entry of $g^P$ is $1$.  In either case $f^P-g^P$ contradicts \ref{argument}.  Therefore, the $j^{\rm th}$ entry of $g^P$ is $1$ and the $l^{\rm th}$ entry of $g^P$ is $-1$. 
Since there are $r-5$ choices for $m$, the number of columns $g^P$ such that $\supp(g^P) \cap \{i,j,k,l\} = \{j,l\}$ is at most $2(r-5)$. This holds for each of the four possible intersections, so the columns $g^P$ such that $|\supp(g^P) \cap \{i,j,k,l\}| = 2$ number at most $8(r-5)$. Therefore, the number of columns of type $3$ is at most $8(r-5) + 24 = 8r - 16$. 

Next, let $g^P$ be a type $2$ column. The support of $g^P$ intersects $\{i,j,k,l\}$, or else $f^P-g^P$ contradicts \ref{argument}. Suppose that $\supp(g^P) = \{i,m\}$ where $m$ is not in $\{1,i,j,k,l\}$. The $i^{\rm th}$ entry of $g^P$ must be $1$ and the $m^{\rm th}$ entry of $g^P$ must be $-1$; otherwise, $f^P-g^P$ contradicts \ref{argument}.  It follows that $f^P+g^P$ has four $-1$'s, contrary to \ref{argument}. Therefore $\supp(g^P) \subseteq \{1,i,j,k,l\}$ so there are at most $24$ columns of type $2$.

If $g^P$ is a column of type $1$, it is straightforward to check that $\supp(g^P) \subseteq \{1,i,j,k,l\}$, so there are at most $8$ columns of type $1$. 

Summarizing the case in which there is a column of type $4$, it now follows from \ref{sublemma_1} and \ref{sublemma_2} that there are at most $ \lfloor\frac{r-1}{2} \rfloor + (8r-34) + (8r-16) + 24 + 8 = \lfloor\frac{33r-37}{2} \rfloor $ columns in $Q-e^P$ so $|E(M)| \leq \lfloor\frac{35r-35}{2} \rfloor$ and the result follows.

\subsubsection*{The Type 3 case.}
Now, suppose that there is no column of type $4$. If there is no column of type $3$, the result follows from \ref{sublemma_4} and \ref{sublemma_5}, so we may assume that there is a column of type $3$. Let $f^P$ be a column of type $3$ with support $\{1,i,j,k\}$. For $\{\alpha, \beta \} = \{1, -1\}$, assume that the $i^{\rm th}$ and $j^{\rm th}$ entries of $f^P$ are $\alpha$ and the $k^{\rm th}$ entry is $\beta$. 

Let $g^P$ be a column of type $2$. The support of $g^P$ intersects $\{i,j,k\}$; otherwise, $f^P-g^P$ would have five non-zero entries, at least three having the same sign, and top entry zero, in contradiction to \ref{argument}. Suppose that $\supp(g^P) = \{1,i,m\}$, where $m$ is not in $\{1,i,j,k\}$. There are at most four columns that have the support $\{1,i,m\}$. However, if the $i^{\rm th}$ entry of $g^P$ is $\beta$ and the $m^{\rm th}$ entry is $\alpha$, then $f^P-g^P$ contradicts \ref{argument}. It follows that there are at most $3(r-4)$ columns in $Q-e^P$ such that $\supp(g^P) \cap \{i,j,k\} = \{i\}$. The number of columns $g^P$ such that $\supp(g^P) \cap \{i,j,k\} = \{j\}$ is also at most $3(r-4)$. Now we consider those columns $g^P$ of type $2$ such that $\supp(g^P) = \{1,k,m\}$, where $m$ is not in $\{1,i,j,k\}$. It follows by \ref{argument} that the $k^{\rm th}$ entry of $g^P$ is $\beta$ and the $m^{\rm th}$ entry is $\alpha$, so there are at most $r-4$ such columns. The number of columns $g^P$ of type $2$ such that $\supp(g^P) \subseteq \{1,i,j,k\}$ is at most $12$, so the number of type $2$ columns in $Q-e^P$ is at most $7(r-4) + 12$. 

For the columns $g^P$ of type $1$ that have $\supp(g^P) \subseteq \{1,i,j,k\}$, there are at most $6$ such columns. If $\supp(g^P) = \{1,m\}$ and $m$ is not in $\{1,i,j,k\}$, then the $m^{\rm th}$ entry is $\alpha$; otherwise, $f^P-g^P$ contradicts \ref{argument}. Let $g^P_1$ and $g^P_2$ be two type $1$ columns whose support is not contained in $\{1,i,j,k\}$. Then $f^P + g^P_1 + g^P_2$ has four entries that are $\alpha$ and has top entry zero, a contradiction. Therefore, there are at most $7$ columns of type $1$ in $Q-e^P$. 

In summary, by \ref{sublemma_1} and \ref{sublemma_3} the number of columns in $Q-e^P$ is at most $\lfloor\frac{r-1}{2} \rfloor + (12r-42) + (7r-16) + 7 =  \lfloor\frac{39r-103}{2} \rfloor$, so $|E(M)| \leq \lfloor\frac{41r-101}{2} \rfloor$. 

\subsubsection*{Conclusion}
This completes the proof that $|E(M)| \leq \lfloor \frac{1}{2}\max(41r-101, 35r-35) \rfloor$ for any $r \geq 5$.  If $r \leq 4$ the bound $\frac{35r-35}{2}$ holds because it exceeds $\frac{3^r-1}{2}$, the maximum size of a ternary matroid of rank $r$.
\end{proof}

\newpage

\section{Two (or more) loose elements in $GF(q)$-matroids}

We partially characterize $GF(q)$-matroids that have two loose elements; then we infer Rajpal's Theorem \ref{4q_bound_intro}.

\begin{theorem}
\label{two_almost_loose}
Let $M$ be a simple $GF(q)$-matroid with no coloops that has two loose elements $e$ and $f$. Then $r(M) \leq 2q$, or $\{e,f\}$ is a cocircuit of $M$.
\end{theorem}

\begin{proof}
We may assume that $\{e,f\}$ is not a cocircuit of $M$. Therefore, there is a basis $B$ that is disjoint from $\{e,f\}$. In the standard $GF(q)$-representation $P = [I_r \mid Q]$, the columns $e^P$ and $f^P$ each have at most one zero. By row and column scaling, we may assume that $e^P$ has all ones except that the first entry may be zero. 

Let $f^{P'}$ denote the column obtained from $f^P$ by removing the first entry. Suppose $f^{P'}$ has at least three non-zero entries that are equal, say at entries $i,j,$ and $k$. Then $ (\{e,f\} \cup B) - \{e_i,e_j,e_k\}$ contains a circuit of $M$ that has size less than $r(M)$ and includes $e$, a contradiction. Since the field $GF(q)$ has $q-1$ non-zero elements, it follows that $f^{P'}$ has at most $2(q-1)$ non-zeroes. Since $f^{P'}$ has at most one zero, $f^{P'}$ has at most $2q-1$ entries. Therefore, $f^P$ has at most $2q$ entries, and thus, the rank of $M$ is at most $2q$.
\end{proof}

\begin{remark}
\label{remark_1}
If $M$ has a spanning circuit containing exactly one of $\{e,f\}$, then the rank bound in Theorem \ref{two_almost_loose} can be improved to be $2q-1$.
\end{remark}

\begin{remark}
\label{remark_ef}
In the case where $\{e,f\}$ is a cocircuit in $M$, so $e$ and $f$ are loose elements but Theorem \ref{two_almost_loose} does not apply, then $M - \{e,f\}$ can be any $GF(q)$-matroid of rank $r(M)-1$.
\end{remark}

\begin{corollary}
\label{two_almost_loose_corollary}
Let $M$ be a simple and cosimple $GF(q)$-matroid that has two loose elements $e$ and $f$. Then $r(M) \leq 2q$. 
\end{corollary}

The following result of Rajpal \cite{rajpal} that slightly improves Theorem \ref{4q_bound_intro} follows immediately from Theorem \ref{two_almost_loose}. 

\begin{corollary}
\label{paving}
Let $M$ be a simple $GF(q)$-paving matroid with no coloops that is not a circuit. Then the rank of $M$ is at most $2q$. Moreover, if $r(M) = 2q$, then $M$ has no spanning circuits.
\end{corollary}

\begin{proof}
Every element of $M$ is loose. If $r(M)>2q$, it follows from Theorem \ref{two_almost_loose} that every pair of elements of $M$ is in a cocircuit of size two. It follows that $M$ is a circuit. Otherwise, $r(M)$ is at most $2q$. If $M$ has a spanning circuit, then by Remark \ref{remark_1}, the rank of $M$ does not exceed $2q-1$. 
\end{proof}

A paving matroid $M$ is \textbf{sparse paving} if its dual $M^*$ is paving. 
A matroid $M$ is sparse paving if and only if each non-spanning circuit of $M$ is a hyperplane. Theorem \ref{4q_bound_intro} due to Rajpal \cite{rajpal} follows immediately from his \cite[Proposition 1$'$]{rajpal}. We provide an alternate proof of the latter. 

\begin{proposition}[{\cite[Proposition 1$'$]{rajpal}}]
\label{sparse_paving}
Let $M$ be a simple $GF(q)$-paving matroid with no coloops. Then $r(M) \leq q$, or $M$ is sparse paving.
\end{proposition}

\begin{proof}
If all the non-spanning circuits of $M$ are hyperplanes, then $M$ is sparse paving; so assume otherwise. Let $C$ be a non-spanning circuit of $M$ that is not a hyperplane. Then there is an element $f$ not in $C$ such that $(C-x) \cup f$ is a circuit for some $x$ in $C$. Pick a basis $B = \{e_1, e_2, \ldots, e_r\}$ that contains $C-x$, say $C-x = \{e_2, e_3 \ldots, e_{r}\}$. Then in the standard $GF(q)$-representation $P = [I_r \mid Q]$, both columns $x^P$ and $f^P$, labeled by elements $x$ and $f$, respectively, have their first entry equal to zero, while all other entries non-zero. By row and column scaling, we may assume that $x^P$ has all ones except at the first entry. Let $f^{P'}$  denote the column obtained from $f^P$ by removing the first entry of $f^P$. Note that all entries of $f^{P'}$ are non-zero. Suppose $f^{P'}$ has at least two entries that are equal, say, entries $i$ and $j$. Then $ (\{x,f\} \cup B) - \{e_1,e_i,e_j\}$ is a circuit of size less than $r(M)$, contrary to assumption. It follows that all entries of $f^{P'}$ are distinct. Since $GF(q)$ has $q-1$ non-zero entries, it follows that $f^P$ has at most $q$ entries. Therefore $r(M) \leq q$. 
\end{proof}

\begin{proof}[Proof of Theorem \ref{4q_bound_intro}]
 By Proposition \ref{sparse_paving}, it follows that $M$ is sparse paving. Since $M$ is not a circuit, by Corollary \ref{paving} the rank of $M$ is at most $2q$. Observe that $M^*$ is not a circuit as $M$ is simple so the rank of $M^*$ does not exceed $2q$. Therefore $|E(M)|$ is at most $4q$. 
\end{proof}

Rajpal \cite{rajpal} introduced paving codes, a generalization of MDS codes, and conjectured that no $(4q,2q)$ paving code exists over $GF(q)$ for $q \geq 4$. Since the matroid represented by the generator matrix of such a code is a $GF(q)$-paving matroid with rank $2q$ and $4q$ elements, the conjecture can be equivalently stated as follows: for $q \geq 4$, there is no $GF(q)$-paving matroid with rank $2q$ and $4q$ elements. 

For $q=2$, the unique binary paving matroid with rank four and eight elements is the binary affine geometry of rank four \cite{acketa}, and for $q=3$, the unique matroid obtained from the Steiner system $S(5,6,12)$ is the ternary paving matroid with rank six and twelve elements \cite{ternary_pav}.

In \cite[Corollary 2.3.1]{thesis}, Rajpal proved the conjecture by showing that a paving matroid with rank $2q$ and $4q$ elements is not representable over $GF(q)$ when $q \geq 4$. It follows that the only matroids attaining the bounds in Theorem \ref{4q_bound_intro} are $AG(3,2)$ for $q=2$, and the matroid obtained from the Steiner system $S(5,6,12)$ for $q=3$.

\end{document}